\documentclass[11pt,oneside, a4paper]{amsart}

\usepackage[english]{babel} 
\usepackage[T1]{fontenc} 
\usepackage{amsmath}
\usepackage{amssymb} 
\usepackage{float}
\usepackage{amsfonts}
\usepackage{mathtools}
\usepackage[utf8]{inputenc}
\usepackage[mathcal]{eucal} 
\usepackage{enumerate}
\usepackage{amsthm} 
\usepackage{hyperref}
\usepackage[totalwidth=14cm,totalheight=20cm]{geometry}

\newtheorem{defi}{Definition}[section] 
\newtheorem{teo}[defi]{Theorem}
\newtheorem{cor}[defi]{Corollary} 
\newtheorem{lemma}[defi]{Lemma}
\newtheorem{prop}[defi]{Proposition}
\newtheorem{oss}[defi]{Remark}

\newcommand{\Pp}{\mathbb{P}}
\newcommand{\R}{\mathbb{R}}

\newcommand{\h}{\mathbb{H}}

\newcommand{\SL}{\mathrm{SL}}

\newcommand{\Ker}{\mathrm{Ker}}

\newcommand{\dPSL}{\mathbb{P}SL(2,\mathbb{R})\times \mathbb{P}SL(2,\mathbb{R})}

\newcommand{\PSL}{\mathbb{P}SL}
\newcommand{\T}{\text{Teich}}
\newcommand{\dT}{\T(S) \times \T(S)}
\newcommand{\Ima}{\mathrm{Im}}
\newcommand{\Isom}{\mathrm{Isom}}

\newcommand{\LHdim}{\mathcal{LH}\mathrm{dim}}
\newcommand{\Imm}{\mathcal{I}m}
\newcommand{\Id}{\mathrm{Id}}
\newcommand{\trace}{\mathrm{trace}}
\newcommand{\length}{\mathrm{length}}

\DeclareMathAlphabet{\mathpzc}{OT1}{pzc}{m}{it}

\title[Degeneration of AdS structures along rays]{Degeneration of globally hyperbolic maximal anti-de Sitter structures along rays}
\author{Andrea Tamburelli}
\date{\today}
\thanks{}

\begin{document}
\maketitle

\begin{abstract} 
Using the parameterisation of the deformation space of GHMC anti-de Sitter structures on $S \times \R$ by the cotangent bundle of the Teichm\"uller space of a closed surface $S$, we study how some geometric quantities, such as the Lorentzian Hausdorff dimension of the limit set, the width of the convex core and the H\"older exponent, degenerate along rays of cotangent vectors.    
\end{abstract}

{\let\thefootnote\relax\footnotetext{{Mathematics Subject Classification: 53A10, 53A35, 53C21, 53C42.}}

\section*{Introduction} 
Anti-de Sitter geometry has grown in interest after the pioneering work of Mess (\cite{Mess}), who pointed out many similarities with hyperbolic geometry and connections with Teichm\"uller theory. Moreover, anti-de Sitter geometry is a useful tool for the study of representations of the fundamental group of a closed, connected, oriented surface $S$ into $\mathbb{P}SO_{0}(2,2) \cong \dPSL$: for instance, representations $\rho=(j, \sigma)$, where $j$ is discrete and faithful and $\sigma$ dominates $j$, are the holonomy representations of closed anti-de Sitter manifolds, diffeomorphic to a circle bundle over $S$ (\cite{Kassel_thesis}, \cite{Salein}), whereas representations $\rho=(\rho_{l}, \rho_{r})$ in which both factors are discrete and faithful correspond to globally hyperbolic maximally Cauchy-compact (GHMC) anti-de Sitter structures on $S \times \R$ (\cite{Mess}). \\
\indent Mess thus deduced that the deformation space of GHMC anti-de Sitter structures on $S \times \R$ can be parameterised by $\dT$. More recently, using the unique maximal surface (i.e. with vanishing mean curvature) embedded in every such manifold (\cite{foliationCMC}, see also \cite{bon_schl}, \cite{toulisse}, \cite{TambuCMC}, \cite{Tambu_regularAdS}, \cite{Tambu_wildAdS} for generalisations), Krasnov and Schlenker found a new parameterisation in terms of the cotangent bundle of the Teichm\"uller space of $S$: a point $(h,q) \in T^{*}\T(S)$ corresponds to the GHMC anti-de Sitter manifold containing a maximal surface whose induced metric is conformal to $h$ and whose second fundamental form is determined by $q$. Since the maximal surface is also a Cauchy surface, in the sense that it intersects every causal curve in exactly one point, Krasnov and Schlenker's parameterisation is an instance of the Choquet-Bruhat theorem for Einstein's equation and the fact that the first and second fundamental form of a Cauchy hypersurface determines the space-time. \\
\\
\indent The main purpose of this paper is to study how relevant quantities that describe the geometry of these manifolds behave along rays of quadratic differentials. The first result is about the Lorentzian Hausdorff dimension of the limit set, recently introduced by Glorieux and Monclair (\cite{GM_Hdim}):\\
\\
{\bf Theorem 3.13.}{\it \ Let $M_{t}$ be the family of GHMC anti-de Sitter manifolds parameterised by $(h,tq_{1}) \in T^{*}\T(S)$, for a fixed non-zero holomorphic quadratic differential $q_{1}$. Then the Lorentzian Hausdorff dimension of the limit set tends to $0$ if $t$ goes to $+\infty$.}\\
\\
We remark that the behaviour of the Lorentzian Hausdorff dimension along other diverging sequences of GHMC anti-de Sitter structures has also been studied in \cite{Olivier_exponent}. However, in this paper we use a completely different approach. For instance, the proofs rely on the comparison between the critical exponent defined through the exponential growth rate of a point in the orbit of $\pi_{1}(S)$ with respect to the Riemannian metric induced on the maximal surface and that defined with respect to a Lorentzian distance on the convex core of $M$, introduced in \cite{GM_Hdim}. In particular, we obtain results about the asymptotic behaviour of the induced metric on the maximal surface:\\
\\
{\bf Proposition 3.10.}{\it \ Let $M_{t}$ be the family of GHMC anti-de Sitter manifolds parameterised by $(h, tq_{1}) \in T^{*}\T(S)$, for a fixed non-zero holomorphic quadratic differential $q_{1}$. Then the induced metric $I_{t}$ on the unique maximal surface embedded in $M_{t}$ satisfies
\[
	\frac{I_{t}}{t} \rightarrow |q_{1}| \ \ \ \  \text{for} \  t \to +\infty
\]
outside the zeros of $q_{1}$, monotonically from above. }\\
\\
From this estimate, we deduce several interesting consequences: the critical exponent for the induced metric on the unique maximal surface embedded in $M_{t}$ strictly decreases to $0$ along the ray (Proposition \ref{prop:entropy_decresce}), and the principal curvatures of the maximal surface have also a precise behaviour: \\
\\
{\bf Corollary 3.11.}{ \it \ The positive principal curvature of the maximal surface embedded in $M_{t}$ monotonically increases to $1$, outside the zeros of $q_{1}$ }. \\
\\
Together with previous results by Seppi (\cite{seppimaximal}), we deduce also the asymptotic behaviour of the width of the convex core of the family $M_{t}$:\\
\\
{\bf Proposition 4.3.}{\it \ The width of the convex core of $M_{t}$ tends to $\pi/2$ if $t$ goes to $+\infty$.}\\
\\
\indent Finally, we introduce another interesting quantity: the H\"older exponent of a GHMC anti-de Sitter manifold. This is defined as follows: if $M$ corresponds to the pair $(\rho_{l}, \rho_{r}) \in \dT$ in Mess' parameterisation, the H\"older exponent $\alpha(M)$ is the best H\"older exponent of a homeomorphism $\phi:\R\Pp^{1} \rightarrow \R\Pp^{1}$ such that $\rho_{r}(\gamma)\circ \phi=\phi \circ \rho_{l}(\gamma)$ for every $\gamma \in \pi_{1}(S)$. In Proposition \ref{prop:formulanuova} we give a geometric interpretation of this quantity: for every $\gamma \in \pi_{1}(S)$, the isometry $\rho(\gamma)=(\rho_{l}(\gamma), \rho_{r}(\gamma))$ leaves two space-like geodesics invariant. On each of them $\rho$ acts by translation and the infimum over all $\gamma$ of the ratio between the difference and the sum of the two translation lengths coincides with the H\"older exponent. We prove the following:\\
\\
{\bf Theorem 2.7.}{\it \ Let $M_{t}$ be the family of GHMC anti-de Sitter manifolds parameterised by $(h,tq_{1}) \in T^{*}\T(S)$. Then $\alpha(M_{t})$ tends to $0$ if $t$ goes to $+\infty$.}\\
\\
\indent It seems more challenging to find coarse estimates for these geometric quantities, in terms of functions depending only on the two points in the Teichm\"uller space of $S$ in Mess' parameterisation, like the Weil-Petersson distance, Thurston's asymmetric distance or $L^{p}$-energies, as obtained for the volume of GHMC anti-de Sitter manifolds in \cite{volumeAdS}. We leave these questions for future work.

\subsection*{Outline of the paper}In Section \ref{sec:background}, we review the basic theory of GHMC anti-de Sitter $3$-manifolds.  The H\"older exponent is studied in Section \ref{sec:holder}. Section \ref{sec:entropy} deals with the behaviour of the Lorentzian Hausdorff dimension. In Section \ref{sec:width} we focus on the width of the convex core.

\section{Background}\label{sec:background}
In this section we recall the basic theory of anti-de Sitter geometry. Good references for the material covered here are \cite{Mess}, \cite{bebo} and \cite{foliationCMC}.\\
\\
\indent Anti-de Sitter space is a Lorentzian manifold diffeomorphic to a solid torus, with constant sectional curvature $-1$. A convenient model for our purposes is the following. Consider the quadratic form $\eta=-\det$ on the vector space $\mathfrak{gl}(2,\R)$ of $2$-by-$2$ matrices. By polarisation, $\eta$ induces a bilinear form of signature $(2,2)$ and its restriction to the submanifold defined by the equation $\eta=-1$ is a Lorentzian metric on $\SL(2, \R)$. Since it is invariant by multiplication by $-\Id$, it defines a Lorentzian metric on 
\[
	\PSL(2,\R)=\SL(2, \R)/\{\pm \Id\} \ .
\]
The projective model of $3$-dimensional anti-de Sitter space $AdS_{3}$ is thus $\PSL(2,\R)$ endowed with this Lorentzian metric. It is orientable and time-orientable, and its group of orientation and time-orientation preserving isometries is 
\[
	\Isom_{0}(AdS_{3})=\PSL(2,\R) \times \PSL(2,\R)\ ,
\]
where the action of an element $(\alpha, \beta) \in \Isom_{0}(AdS_{3})$ is given by
\[
	(\alpha, \beta)\cdot \gamma= \alpha\gamma\beta^{-1} \ \ \ \forall \ \gamma \in AdS_{3} \ .
\]
The boundary at infinity of $AdS_{3}$ is defined as the projectivisation of rank $1$ matrices:
\[
	\partial_{\infty}AdS_{3}=\Pp\{ \alpha  \in \mathfrak{gl}(2, \R) \ | \ \det(\alpha)=0, \ \alpha \neq 0 \} \ .
\]
We can identify $\partial_{\infty}AdS_{3}$ with $\R\Pp^{1}\times \R\Pp^{1}$ in the following way: an equivalence class of a rank $1$ matrix $M$ is sent to the pair of lines $(\Ima(M), \Ker(M))$. In this way, the action of $\Isom_{0}(AdS_{3})$ extends to the boundary at infinity and corresponds to the obvious action of $\dPSL$ on $\R\Pp^{1}\times \R\Pp^{1}$.\\
\\
Thinking of $AdS_{3} \subset \R\Pp^{3}$, geodesics are obtained as intersection between projective lines and $AdS_{3}$. 
In particular, geodesics through $\Id \in \PSL(2, \R)$ are $1$-parameter subgroups.\\
\\
\indent Projective duality (or, equivalently, orthogonality with respect to the quadratic form $\eta$) induces a duality between points and space-like planes in $AdS_{3}$. Given a space-like plane $P$, the time-like geodesics orthogonal to $P$ intersect at a point $P^{*}$ after a time $\pi/2$. Viceversa, given a point $p$, every time-like geodesic starting at $p$ intersects a unique space-like plane orthogonally after a time $\pi/2$. 

\subsection{GHMC anti-de Sitter manifolds}
We are interested in a special class of manifolds locally isometric to $AdS_{3}$. \\
\\
\indent We say that an anti-de Sitter three manifold $M$ is Globally Hyperbolic Maximally Cauchy-compact (GHMC) if it contains an embedded closed, oriented surface $S$ that intersects every inextensible causal curve in exactly one point, and if $M$ is maximal by isometric embeddings sending a Cauchy surface to a Cauchy surface. It turns out that $M$ is necessarily diffeomorphic to a product $S \times \R$ (\cite{MR0270697}). Moreover, we will assume throughout the paper that $S$ has genus $\tau\geq 2$. \\
\\
\indent We denote with $\mathcal{GH}(S)$ the deformation space of GHMC anti-de Sitter structures on $S\times \R$, i.e. the space of maximal globally hyperbolic anti-de Sitter metrics on $S\times \R$ up to diffeomorphisms isotopic to the identity.
\begin{teo}[\cite{Mess}] $\mathcal{GH}(S)$ is parameterised by $\T(S) \times \T(S)$.
\end{teo}

The diffeomorphism is constructed as follows. Given a GHMC anti-de Sitter structure, its holonomy representation $\rho:\pi_{1}(S) \rightarrow \Isom(AdS_{3})$ induces a pair of representations $(\rho_{l}, \rho_{r})$ by projecting onto each factor. Mess proved that both are faithful and decrete and thus define two points in $\T(S)$. On the other hand, given a pair of Fuchsian representations $(\rho_{l}, \rho_{r})$, there exists a unique homeomorphism $\phi: \R\Pp^{1} \rightarrow \R\Pp^{1}$ such that $\rho_{r}(\gamma)\circ \phi=\phi \circ \rho_{l}(\gamma)$ for every $\gamma \in \pi_{1}(S)$. The graph of $\phi$ defines a curve $\Lambda_{\rho}$ on the boundary at infinity of $AdS_{3}$ and Mess contructed a maximal domain of discontinuity $\mathcal{D}(\phi)$ for the action of $\rho(\pi_{1}(S)):=(\rho_{l}, \rho_{r})(\pi_{1}(S))$, called domain of dependence, by considering the set of points whose dual space-like plane is disjoint from $\Lambda_{\rho}$. The quotient
\[
	M=\mathcal{D}(\phi)/\rho(\pi_{1}(S))
\]
is the desired GHMC anti-de Sitter manifold. \\
\\
\indent Manifolds corresponding to the diagonal in $\dT$ are called Fuchsian. In this case, the homeomorphism $\phi$ is the identity and the corresponding curve on the boundary at infinity of $AdS_{3}$ is the boundary of the totally geodesic space-like plane
\[
	P_{0}=\{ \alpha \in \PSL(2, \R) \ | \ \trace(\alpha)=0\} \ .
\]
The representation $\rho=(\rho_{0}, \rho_{0})$ preserves $P_{0}$ and, thus a foliation by equidistant surfaces from $P_{0}$. The quotient
\[
	M_{F}=\mathcal{D}(\Id)/\rho(\pi_{1}(S))
\]
is thus isometric to 
\[
	\left( S\times \left(-\frac{\pi}{2}, \frac{\pi}{2}\right), g_{F}=\cos^{2}(t)h_{0}-dt^{2} \right) \ ,
\]
where $h_{0}$ is the hyperbolic metric with holonomy $\rho_{0}$. \\
\\
\indent Mess introduced also the notion of convex core. This is the smallest convex subset of a GHMC anti-de Sitter manifold $M$ which is homotopically equivalent to $M$. It can be concretely realised as follows. If $\rho$ denotes the holonomy representation of $M$ and $\Lambda_{\rho}\subset \partial_{\infty}AdS_{3}$ is the limit set of the action of $\rho(\pi_{1}(S))$, the convex core of $M$ is
\[
	\mathcal{C}(M)=\mathcal{C}(\Lambda_{\rho})/\rho(\pi_{1}(S)) \ ,
\]
where $\mathcal{C}(\Lambda_{\rho})$ denotes the convex-hull of the curve $\Lambda_{\rho}$ in $AdS_{3}$. \\
If $M$ is Fuchsian, the convex core is a totally geodesic surface. Otherwise, it is a three-dimensional domain, homeomorphic to $S \times [0,1]$, the two boundary components being space-like surfaces, endowed with hyperbolic metrics and pleated along measured laminations.  

\subsection{A parameterisation using maximal surfaces}\label{subsec:para_max}
In this paper we use another parameterisation of the deformation space of GHMC anti-de Sitter structures on $S \times \R$, introduced by Krasnov and Schlenker (\cite{Schlenker-Krasnov}). We recall here the main steps of their construction.\\
\\
\indent Let $M$ be a GHMC anti-de Sitter $3$-manifold. It is well-known (\cite{foliationCMC}) that $M$ contains a unique embedded maximal surface $\Sigma$, i.e. with vanishing mean curvature. By the Fundamental Theorem of surfaces embedded in anti-de Sitter space, $\Sigma$ is uniquely determined by its induced metric $I$ and its shape operator $B:T\Sigma \rightarrow T\Sigma$, which are related to each other by the Gauss-Codazzi equations:
\begin{align*}
	&d^{\nabla_{I}}B=0 \\	
	&K_{I}=-1-\det(B) \ ,
\end{align*}
where we have denoted with $K_{I}$ the curvature of the metric $I$. The first equation, together with the fact that $B$ is traceless, implies that the second fundamental form $II=I(B\cdot, \cdot)$ is the real part of a quadratic differential $q$ (\cite{Hopf}), which is holomorphic for the complex structure compatible with the metric, in the following sense. For every pair of vector fields $X$ and $Y$ on $\Sigma$, we have
\[
	\Re(q)(X,Y)=I(BX, Y) \ .
\]
In a local conformal coordinate $z$, we can write $q=f(z)dz^{2}$ and $I=e^{2u}|dz|^{2}$. Thus, $\Re(q)$ is the bilinear form that in the frame $\{\partial_{x}, \partial_{y}\}$ is represented by
\[
	\Re(q)=\begin{pmatrix}
		\Re(f) & -\Imm(f)\\
		-\Imm(f) & -\Re(f)
		\end{pmatrix} \ ,
\]
and the shape operator $B$ can be recovered as $B=I^{-1}\Re(q)$. Moreover, a straightforward computation shows that the Codazzi equation for $B$ is equivalent to the Cauchy-Riemann equations for $f$. Therefore, we can define a map
\begin{align*}
	\Psi:\mathcal{GH}(S) &\rightarrow T^{*}\T(S)\\
		M &\mapsto (h,q)
\end{align*}
associating to a GHMC anti-de Sitter structure the unique hyperbolic metric in the conformal class of $I$ and the quadratic differential $q$, constructed from the embedding data of the maximal surface $\Sigma$ embedded in $M$. \\
\\
\indent In order to prove that $\Psi$ is a homeomorphism, Krasnov and Schlenker (\cite{Schlenker-Krasnov}) found an explicit inverse. They showed that, given a hyperbolic metric $h$ and a quadratic differential $q$ that is holomorphic for the complex structure compatible with $h$, it is always possible to find a smooth map $u: S \rightarrow \R$ such that $I=e^{2u}h$ and $B=I^{-1}\Re(q)$ are the induced metric and the shape operator of a maximal surface embedded in a GHMC anti-de Sitter manifold. This is accomplished by noticing that the Codazzi equation for $B$ is trivially satisfied since $q$ is holomorphic, and thus it is sufficient to find $u$ so that the Gauss equation holds. Now, 
\[
	\det(B)=\det(e^{-2u}h^{-1}\Re(q))=e^{-4u}\det(h^{-1}\Re(q))=-e^{-4u}\|q\|^{2}_{h}
\]
and
\[
	K_{I}=e^{-2u}(K_{h}-\Delta_{h}u) ,
\]
hence the Gauss equation translates into the quasi-linear PDE
\begin{equation}\label{eq:PDE}
	\Delta_{h}u=e^{2u}-e^{-2u}\|q\|_{h}^{2}+K_{h} \ .
\end{equation}

\begin{prop}[Lemma 3.6 \cite{Schlenker-Krasnov}]\label{prop:existence-sol} There exists a unique smooth solution $u:S \rightarrow \R$ to Equation (\ref{eq:PDE}). Moreover, for any non-trivial holomorphic quadratic differential $q_{1}$, along the ray $q=tq_{1}$, the solution $u_{t}$ depends smoothly on $t\in \R$.
\end{prop}
\begin{proof} Existence and uniqueness were proved in \cite{Schlenker-Krasnov}. Let us now prove the smooth dependence of the solution on $t$. Consider the mapping $F:C^{2,\alpha}(S)\times \R \rightarrow C^{0,\alpha}(S)$ defined by 
\[
 F(w, t)=\Delta_{h}w-e^{2w}+e^{-2w}\|tq_{1}\|_{h}^{2}-K_{h}  \ .
\]
Notice that the solution $u_{t}$ to Equation (\ref{eq:PDE}) along the ray $q=tq_{1}$ satisfies $F(u_{t},t)=0$. At any $t_{0} \in \R$, the linearization  
\begin{align*}
  dF_{w}&(u_{t_{0}},t_{0}): C^{2,\alpha}(S) \rightarrow C^{0,\alpha}(S) \\
      &\psi \mapsto \Delta_{h}\psi-(2e^{2u_{t_{0}}}+2e^{-2u_{t_{0}}}\|t_{0}q_{1}\|_{h}^{2})\psi
\end{align*}
is a negative definite operator, hence invertible. By the Implicit Function Theorem for Banach spaces, the solution $u_{t}$ to Equation (\ref{eq:PDE}) depends smoothly on $t$ in a neighbourhood of $t_{0}$.
\end{proof}

In Section \ref{sec:entropy}, we will give precise estimates for the solution $u$ in terms of the quadratic differential $q$, and study its asymptotics along a ray $q=tq_{1}$ for a fixed non-trivial holomorphic quadratic differential $q_{1}$. 

\subsection{Relation between the two parameterisations} The theory of harmonic maps between hyperbolic surfaces provides a bridge between the two parameterisations of $\mathcal{GH}(S)$. 

\begin{defi}A diffeomorphism $m:(S,h_{l}) \rightarrow (S,h_{r})$ is minimal Lagrangian if it is area-preserving and its graph is a minimal surface in $(S\times S, h\oplus h')$.
\end{defi}

Minimal Lagrangian maps between hyperbolic surfaces have been extensively studied (\cite{oneharmonic}, \cite{seppimaximal}, \cite{Tambu_polygons}, \cite{Tambu_regularAdS}). We will use, in particular, the following fundamental result:

\begin{teo}[\cite{labourieCP}, \cite{Schoenharmonic}] Given two hyperbolic metrics $h_{l}$ and $h_{r}$ on a closed surface $S$, there exists a unique minimal Lagrangian map $m:(S, h_{l}) \rightarrow (S,h_{r})$ isotopic to the identity.
\end{teo}

This should remind the reader of the more classical existence and uniqueness of harmonic maps between hyperbolic surfaces. Recall that a map $f:(N,g) \rightarrow (N',g')$ between Riemannian manifolds is harmonic if it is a critical point of the Dirichlet energy functional
\[
		E(f)=\int_N \| df\|^{2} dV_{g} \ . 
\]
If $(N',g')$ is non-positively curved, harmonic maps are actually local minimizers and are unique in a fixed homotopy class. Moreover, in the special case of surface domains, the energy functional is conformally invariant, thus harmonicity depends only on the conformal class of the Riemannian metric on the domain. \\
\indent If we specialize to harmonic maps between surfaces $f:(S,h) \rightarrow (S,h')$ and we fix a complex structure on $S$ compatible with the metric $h$, we can define the Hopf differential of $f$ as $\mathrm{Hopf}(f)=(f^{*}h')^{(2,0)}$. The map $f$ being harmonic implies that $\mathrm{Hopf}(f)$ is a holomorphic quadratic differential on $S$ that measures how far $f$ is from being conformal. \\

\indent Minimal Lagrangian diffeomorphisms are closely related to harmonic maps in the following sense: a minimal Lagrangian diffeomorphism $m:(S,h_{l}) \rightarrow (S,h_{r})$ can be factorised (\cite{bon_schl}) as $m=f'\circ f^{-1}$, where
\[
	f:(S,h) \rightarrow (S,h_{l}) \ \ \ \ \text{and} \ \ \ \ f':(S,h) \rightarrow (S,h_{r})
\]
are harmonic with opposite Hopf differentials. We call $h$ the center of the minimal Lagrangian map.

\begin{prop}[\cite{bon_schl}]\label{prop:rel_para} Let $h_{r}$ and $h_{l}$ be hyperbolic metrics on $S$ with holonomy $\rho_{r}$ and $\rho_{l}$. The center of the minimal Lagrangian map $m:(S, h_{l}) \rightarrow (S, h_{r})$ is the conformal class of the induced metric on the maximal surface $\Sigma$ contained in the GHMC anti-de Sitter manifold $M$ with holonomy $\rho=(\rho_{l}, \rho_{r})$. Moreover, the second fundamental form of $\Sigma$ is (up to a factor $\pm i$) the real part of the Hopf differential of the harmonic map factorising $m$.
\end{prop}

This picture has been recently generalised to hyperbolic surfaces with cone singularities (\cite{toulisse}) and to other families of diffeomorphisms between hyperbolic surfaces, called landslides (\cite{bsads}, \cite{Tambu_Qiyu}).

\section{H\"older exponent}\label{sec:holder}
In this section we introduce the H\"older exponent of a GHMC anti-de Sitter manifold and study its asymptotic behaviour along a ray of quadratic differentials.\\
\\
\indent Let $M$ be a GHMC anti-de Sitter manifold. Its holonomy representation $\rho:\pi_{1}(S) \rightarrow \dPSL$ gives rise, by projecting into each factor, to two discrete and faithful representations $\rho_{l}$ and $\rho_{r}$.
Let $\phi:\R\Pp^{1} \rightarrow \R\Pp^{1}$ be the unique homeomorphism such that
\[
	\rho_{r}(\gamma)\circ \phi=\phi \circ \rho_{l}(\gamma) \ \ \ \ \ \text{for every} \ \gamma \in \pi_{1}(S)  \ .
\]
It is well-known (\cite{Thurstondistance}) that $\phi$ is quasi-symmetric, and, in particular, has H\"older regularity.

\begin{defi}\label{defi:Hol}Given a H\"older map $f:\R\Pp^{1} \rightarrow \R\Pp^{1}$, we define the H\"older exponent of $f$ as
\[
  \alpha(f)=\sup\{ \alpha \in (0,1] \ | \ \exists C>0 \ d_{\R\Pp^{1}}(x,y)\leq Cd_{\R\Pp^{1}}(x,y)^{\alpha} \ \ \forall  x,y\in \R\Pp^{1} \} \ .
\]
The H\"older exponent $\alpha(M)$ of $M$ is the minimum between the H\"older exponents of $\phi$ and $\phi^{-1}$.
\end{defi}

\begin{oss}This definition takes into account that $\phi$ and $\phi^{-1}$ have in general different H\"older exponents. On the other hand, the manifolds with holonomies $(\rho_{l}, \rho_{r})$ and $(\rho_{r}, \rho_{l})$ are isometric, because the map
\begin{align*}
	\PSL(2,\R) &\rightarrow \PSL(2,\R)\\
	A &\mapsto A^{-1}
\end{align*}
induces an orientation-reversing isometry of $AdS_{3}$ which swaps the left and right holonomies in Mess' parameterisation. Hence, we expect a geometric interesting quantity to be invariant under this transformation.
\end{oss}

An explicit formula for the H\"older exponent of $\phi$ is well-known:
\begin{teo}[Chapter 7 Proposition 14 \cite{Hyp90}, Theorem 6.5 \cite{BS11}]\label{teo:formulavecchia}Let $\rho_{r}$ and $\rho_{l}$ be Fuchsian representations. The H\"older exponent of the unique homeomorphism $\phi:\R\Pp^{1}\rightarrow \R\Pp^{1}$ such that
\[
	\rho_{r}(\gamma)\circ \phi=\phi\circ \rho_{l}(\gamma) \ \ \ \ \ \text{for every} \ \ \gamma \in \pi_{1}(S)
\]
is
\[
	\alpha(\phi)=\inf_{\gamma \in \pi_{1}(S)}\frac{\ell_{r}(\gamma)}{\ell_{l}(\gamma)}
\]
where $\ell_{r}(\gamma)$ and $\ell_{l}(\gamma)$ denote the lengths of the geodesic representatives of $\gamma$ with respect to the hyperbolic metrics with holonomy $\rho_{r}$ and $\rho_{l}$, respectively.
\end{teo}

\indent Therefore, the H\"older exponent of a GHMC anti-de Sitter manifold with holonomy $\rho=(\rho_{l}, \rho_{r})$ is given by
\begin{equation}\label{eq:formula_HolM}
	\alpha(M)=\inf_{\gamma \in \pi_{1}(S)}\min\left\{\frac{\ell_{r}(\gamma)}{\ell_{l}(\gamma)}, \frac{\ell_{l}(\gamma)}{\ell_{r}(\gamma)}\right\} \ .
\end{equation}

\begin{oss}\label{oss}Since the formula for $\alpha(M)$ is homogeneous and weighted simple closed curves are dense in the space of measured foliations, the above formula is equivalent to
\[
	\alpha(M)=\inf_{\mu \in \mathcal{MF}(S)}\min\left\{\frac{\ell_{r}(\mu)}{\ell_{l}(\mu)}, \frac{\ell_{l}(\mu)}{\ell_{r}(\mu)}\right\} \ .
\]
\end{oss}

\indent We easily deduce a rigity property of the H\"older exponent:

\begin{prop}[Mostow]The H\"older exponent of a GHMC anti-de Sitter manifold is equal to $1$ if and only if $M$ is Fuchsian.
\end{prop}
\begin{proof} The result can be deduced from the proof of Mostow's rigidity \cite{Mostowbook}. We provide here a brief argument based on Thurston's work on Lipschitz maps between hyperbolic surfaces (\cite{Thurstondistance}).\\
\indent If $M$ is Fuchsian $\ell_{r}(\gamma)=\ell_{l}(\gamma)$ for every $\gamma \in \pi_{1}(S)$, hence the H\"older exponent is equal to $1$. On the other hand, if $M$ is not Fuchsian, there exists a curve $\gamma \in \pi_{1}(S)$ such that $\ell_{l}(\gamma)>\ell_{r}(\gamma)$, hence $\alpha(M)<1$.
\end{proof}

\indent Before studying the asymptotics of the H\"older exponent along rays of quadratic differentials, we want to give a new interpretation of the H\"older exponent that is more related to anti-de Sitter geometry. \\
\indent Let $\rho=(\rho_{r}, \rho_{l})$ be the holonomy representation of a GHMC anti-de Sitter structure. Let us suppose first that $\rho_{l}\neq \rho_{r}$. Since $\rho_{l}$ and $\rho_{r}$ are the holonomies of hyperbolic structures on $S$, for every $\gamma \in \pi_{1}(S)$, the elements $\rho_{l}(\gamma)$ and $\rho_{r}(\gamma)$ are hyperbolic isometries of the hyperbolic plane. Therefore, there exist $A,B \in \PSL(2,\R)$ such that
\[
	A\rho_{l}(\gamma)A^{-1}=\begin{pmatrix}
				e^{\ell_{l}(\gamma)/2} & 0 \\
					0 & e^{-\ell_{l}(\gamma)/2}
				\end{pmatrix} \ \ \ \ \ \ \ \ \ \ B\rho_{r}(\gamma)B^{-1}=\begin{pmatrix}
									e^{\ell_{r}(\gamma)/2} & 0\\
										0 & e^{-\ell_{r}(\gamma)/2} 
												\end{pmatrix}\ .
\]
We thus notice that the isometry of $AdS_{3}$ given by $\rho(\gamma)=(\rho_{l}(\gamma), \rho_{r}(\gamma))$ leaves two space-like geodesics invariant
\[
	\sigma^{*}(t)=A\begin{pmatrix}
			e^{t} & 0 \\
			0 & e^{-t} 
			\end{pmatrix}B^{-1} \ \  \ \ \ \ \  \text{and} \ \ \ \ \ \ \sigma(t)=A\begin{pmatrix}
											0 & e^{t} \\
											e^{-t} & 0 
											\end{pmatrix}B^{-1} \ .
\]
An easy computation shows that the isometry $\rho(\gamma)$ acts on $\sigma^{*}$ by translation with translation length 
\[
	\beta^{*}(\gamma)=\frac{|\ell_{l}(\gamma)-\ell_{r}(\gamma)|}{2}
\]
and acts by translation on $\sigma$ with translation length 
\[
	\beta(\gamma)=\frac{\ell_{l}(\gamma)+\ell_{r}(\gamma)}{2} \ .
\]
\indent We claim that only the geodesic $\sigma$ is contained in the convex hull of the limit set $\Lambda_{\rho}$. Recall that the limit set can be constructed as the graph of the homeomorphism $\phi: \R\Pp^{1} \rightarrow \R\Pp^{1}$ (\cite{Mess}) such that
\[
	\rho_{r}(\gamma)\circ \phi=\phi \circ \rho_{l}(\gamma) \ \ \ \ \ \text{for every} \ \ \ \gamma \in \pi_{1}(S) \ .
\]
In particular, $\phi$ sends the attactive (resp. repulsive) fixed point of $\rho_{l}(\gamma)$ into the attractive (resp. repulsive) fixed point of $\rho_{r}(\gamma)$. Therefore, we must have
\[
	\phi(A[1:0])=B[1:0] \ \ \ \ \text{and} \ \ \ \ \phi(A[0:1])=B[0:1] \ .
\]	
Now, the geodesic $\sigma$ has ending points
\[
	\sigma(-\infty)=(A[0:1],B[0:1]) \in \R\Pp^{1} \times \R\Pp^{1}
\]
and
\[
	 \sigma(+\infty)=(A[1:0], B[1:0]) \in \R\Pp^{1} \times \R\Pp^{1} \ ,
\]
whereas the geodesic $\sigma^{*}$ has ending points
\[
	\sigma^{*}(-\infty)=(A[0:1],B[1:0]) \in \R\Pp^{1} \times \R\Pp^{1}
\]
and
\[
	 \sigma^{*}(+\infty)=(A[1:0], B[0:1]) \in \R\Pp^{1} \times \R\Pp^{1}
\]
hence only the ending points of $\sigma$ lie on the limit curve $\Lambda_{\rho}$. As a consequnce, $\sigma$ is contained in the convex hull of $\Lambda_{\rho}$ and its projection is a closed space-like geodesic in the convex core of $M$. On the other hand, the geodesic $\sigma^{*}$ does not even belong to the domain of dependence of $\Lambda_{\rho}$. In fact, it it easy to check that the dual space-like plane of any point of $\sigma^{*}$ contains the geodesic $\sigma$, thus its boundary at infinity is not disjoint from the limit curve $\Lambda_{\rho}$.\\
\\
\indent In the special case, when $\rho_{r}=\rho_{l}$, the point $[Id] \in AdS_{3}$ is fixed and its dual space-like plane $P_{0}$ is left invariant. By definition of the dual plane (see Section \ref{sec:background}),
\[
	P_{0}=\{ A \in \PSL(2, \R) \ | \ \trace(A)=0 \}
\]
is the dual of $[Id] \in AdS_{3}$ and it is easy to check that it is a copy of the hyperbolic plane. With this identification, $\rho(\gamma)$ acts on $P_{0}$ as the hyperbolic isometry $\rho_{r}(\gamma)=\rho_{l}(\gamma)$ does on $\h^{2}$. \\
\\
\indent We thus obtain another way of computing the H\"older exponent of a GHMC anti-de Sitter manifold:
\begin{prop}\label{prop:formulanuova} Let $M$ be a GHMC anti-de Sitter manifold with holonomy $\rho$. Let $\beta(\gamma)$ and $\beta^{*}(\gamma)$ be the translation lengths of the isometries $\rho(\gamma)$ for every $\gamma \in \pi_{1}(S)$. Then
\[
	\alpha(M)=\inf_{\gamma \in \pi_{1}(S)}\frac{\beta(\gamma)-\beta^{*}(\gamma)}{\beta(\gamma)+\beta^{*}(\gamma)} \ .
\]
\end{prop}
\begin{proof}This is a direct consequence of the explicit formulas for $\beta(\gamma)$ and $\beta^{*}(\gamma)$ and Theorem \ref{teo:formulavecchia}.
\end{proof}

\indent We can now describe the asymptotic behaviour of the H\"older exponent:
\begin{teo}Let $M_{t}$ be the family of GHMC anti-de Sitter manifolds parameterised by the ray $(h,tq_{1}) \in T^{*}\T(S)$ for a non-zero quadratic differential $q_{1}$. Then 
\[
	\lim_{t \to +\infty}\alpha(M_{t})=0 \ .
\]
\end{teo}
\begin{proof}Let $\rho_{t}=(\rho_{l,t}, \rho_{r,t})$ be the holonomy representation of $M_{t}$. Let $h_{l,t}$ and $h_{r,t}$ be the hyperbolic metrics on $S$ with holonomy $\rho_{l,t}$ and $\rho_{r,t}$, respectively. By Proposition \ref{prop:rel_para}, we can suppose that the identity maps
\[
	id:(S, h) \rightarrow (S, h_{l, t}) \ \ \ \ \ id:(S,h) \rightarrow (S, h_{r,t})
\]
are harmonic with Hopf differentials $itq_{1}$ and $-itq_{1}$, respectively.\\
Associated to $itq_{1}$ are two measured foliations $\lambda^{+}_{t}$ and $\lambda^{-}_{t}$: in a natural conformal coordinate $z=x+iy$ outside the zeros of $iq_{1}$, we can express $itq_{1}=dz^{2}$. The foliations are then given by
\[
	\lambda^{+}_{t}=(y=const, z^{*}|dy|) \ \ \ \ \text{and} \ \ \ \ \ \lambda^{-}_{t}=(x=const, z^{*}|dx|) \ .
\]
Notice, in particular, that the support of the foliation is fixed for every $t>0$ and only the measure changes, being it multiplied by $t^{1/2}$. We can thus write
\[
	\lambda^{+}_{t}=t^{1/2}\lambda^{+}_{1} \ \ \ \ \text{and} \ \ \ \ \lambda^{-}_{t}=t^{1/2}\lambda^{-}_{1}
\]
where $\lambda^{\pm}_{1}$ are the measured foliations associated to $iq_{1}$. Moreover, multiplying a quadratic differential by $-1$ interchanges the two foliations. \\
\indent By Wolf's compactification of Teichm\"uller space (Section 4.2 \cite{Wolf_harmonic}), we know that
\[
	\lim_{t \to +\infty}\frac{\ell_{l,t}(\gamma)}{2t^{1/2}}=\iota(\lambda_{1}^{+}, \gamma)
\]
for every $\gamma \in \pi_{1}(S)$, where $\ell_{l,t}(\gamma)$ denotes the length of the geodesic representative of $\gamma$ with respect to the hyperbolic metric $h_{l,t}$. By density, the same holds for every measured foliation on $S$. Therefore, using Remark \ref{oss}, 
\begin{align*}
	0\leq \limsup_{t\to +\infty}\alpha(M_{t})&=\limsup_{t\to +\infty} \inf_{ \mu \in \mathcal{MF}(S)}\min\left\{\frac{\ell_{l,t}(\mu)}{\ell_{r,t}(\mu)}, \frac{\ell_{r,t}(\mu)}{\ell_{l,t}(\mu)}\right\} \\
		&\leq \limsup_{t \to +\infty}\frac{\ell_{l,t}(\lambda_{1}^{+})}{\ell_{r,t}(\lambda_{1}^{+})}
		=\limsup_{t \to +\infty}\frac{\ell_{l,t}(\lambda_{1}^{+})}{2t^{1/2}}\frac{2t^{1/2}}{\ell_{r,t}(\lambda_{1}^{+})}\\
		&=\frac{\iota(\lambda_{1}^{+}, \lambda_{1}^{+})}{\iota(\lambda_{1}^{-}, \lambda_{1}^{+})}=0
\end{align*}
because every measured foliation has vanishing self-intersection and $\iota(\lambda_{1}^{-}, \lambda_{1}^{+})\neq 0$ by construction, since $\lambda_{1}^{\pm}$ are the horizontal and vertical foliations of a quadratic differential.
\end{proof}

\section{Critical exponents}\label{sec:entropy}
In this section we study the asymptotic behaviour of the Lorentzian Hausdorff dimension of the limit curve $\Lambda_{\rho}$ associated to a GHMC anti-de Sitter manifold. 

\subsection{Lorentzian Hausdorff dimension}
Let $M$ be a GHMC anti-de Sitter manifold with holonomy representation $\rho$. In Section \ref{sec:background}, we saw that the limit set of the action of $\rho(\pi_{1}(S))$ is a simple closed curve $\Lambda_{\rho}$ in the boundary at infinity of $AdS_{3}$. Moreover, $\Lambda_{\rho}$ is the graph of a locally Lipschitz function, thus its Hausdorff dimension is always $1$. Recently, Glorieux and Monclair defined a notion of Lorenztian Hausdorff dimension, that manages to describe how far the representiation $\rho$ is from being Fuchsian. This resembles the usual definition of Hausdorff dimension, where instead of considering coverings consisting of Euclidean balls, they used Lorentzian ones (\cite[Section 5.1]{GM_Hdim}). They also gave an equivalent definition in terms of critical exponent related to a distance-like function in $AdS_{3}$.

\begin{defi}Let $\Lambda_{\rho}\subset \partial_{\infty}AdS_{3}$ be the limit set of the holonomy of a GHMC anti-de Sitter structure. The Lorentzian distance 
\[
	d_{AdS}: \mathcal{C}(\Lambda_{\rho}) \times \mathcal{C}(\Lambda_{\rho}) \rightarrow \R_{\geq 0}
\]
is defined as follows. Let $x,y \in \mathcal{C}(\Lambda_{\rho})$ and let $\gamma_{x,y}$ be the unique geodesic connecting $x$ and $y$. We put
\[
	d_{AdS}(x,y):=\begin{cases}
			\length(\gamma_{x,y}) \ \ \ \ \ \text{if $\gamma_{x,y}$ is space-like} \\
			0 \ \ \ \ \ \ \ \ \  \ \ \ \ \ \ \ \ \ \text{otherwise}
			  \end{cases}
\]
\end{defi}
The function $d_{AdS}$ is a distance-like function in the following sense: it is symmetric, and there exists a constant $k_{\rho}$ depending on the representation $\rho$ such that
\[
  d_{AdS}(x,z) \leq d_{AdS}(x,y)+d_{AdS}(y,z)+k_{\rho}
\]
for every $x,y,z \in \mathcal{C}(\Lambda_{\rho})$ (\cite[Theorem 3.4]{GM_Hdim}). 

\begin{defi}The critical exponent of $\rho(\pi_{1}(S))$ relative to the Lorentzian distance $d_{AdS}$ is 
\[
	\delta_{AdS}(\rho)=\limsup_{R \to +\infty}\frac{1}{R}\log(\#\{\gamma \in \pi_{1}(S) \ | \ d_{AdS}(\rho(\gamma)x_{0}, x_{0}) \leq R\}) \ ,
\]
where $x_{0}\in \mathcal{C}(\Lambda_{\rho})$ is a fixed base point.
\end{defi}

\begin{oss} By the above weak triangle inequality, the critical exponent $\delta_{AdS}(\rho)$ does not depend on the choice of the basepoint $x_{0}$. 
\end{oss}

The link between the critical exponent for the Lorentzian distance $d_{AdS}$ and the Lorentzian Hausdorff dimension is provided by the following result:
\begin{teo}[Theorem 1.1 \cite{GM_Hdim}]\label{thm:equivalence}Let $\Lambda_{\rho}$ be the limit set of the holonomy representation $\rho$ of a GHMC anti-de Sitter structure. Then
\[
	\LHdim(\Lambda_{\rho})=\delta_{AdS}(\rho) \ .
\]
\end{teo}

\subsection{Critical exponent of the maximal surface}
Another natural quantity that can be associated to a GHMC anti-de Sitter structure is the critical exponent that computes the exponential growth rate of a point in the orbit of $\pi_{1}(S)$ with respect to the Riemannian metric induced on the unique maximal surface. We will use this in the next subsection to provide an upper-bound for the Lorentzian Hausdorff dimension of the limit set.\\
\\
\indent Let $g$ be a Riemannian metric or a flat metric with conical singularity on the surface $S$. Let $\tilde{S}$ be the universal cover of $S$. The critical exponent of $\pi_{1}(S)$ relative to $g$ can be defined as
\[
	\delta(g)=\limsup_{R \to +\infty}\frac{1}{R}\log(\#\{ \gamma \in \pi_{1}(S) \ | \ d_{g}(\gamma\cdot x_{0}, x_{0})\leq R \}) \in \R^{+}
\]
where $x_{0} \in \tilde{S}$ is an arbitrary base point. \\
\\
\indent We introduce the function $\delta:T^{*}\T(S) \rightarrow \R$ that associates to a point $(h,q) \in T^{*}\T(S)$ the critical exponent relative to the Riemannian metric $I=e^{2u}h$, where $u$ is the solution to Equation (\ref{eq:PDE}). Namely, $\delta(h,q)$ is the critical exponent relative to the Riemannian metric induced on the unique maximal surface embedded in the GHMC anti-de Sitter manifold corresponding to $(h,q)$. By identifying $T^{*}\T(S)$ with $\mathcal{GH}(S)$ (see Section \ref{sec:background}), we will often denote this map as $\delta(\rho)$, where $\rho$ is the holonomy representation of the corresponding GHMC anti-de Sitter structure.\\
\\
\indent Notice that, since in Equation (\ref{eq:PDE}) only the $h$-norm of the quadratic differential $q$ appears, the function $\delta$ is invariant under the natural $S^{1}$ action on $T^{*}\T(S)$ given by $(h,q) \mapsto (h, e^{i\theta}q)$. 

\subsection{Estimates for the induced metric on the maximal surface}
In this section we study the asymptotic behaviour of the induced metric $I_{t}$ on the maximal surface $\Sigma_{t}$ along a ray $tq_{1}$ of quadratic differentials. We deduce also estimates for the principal curvatures of $\Sigma_{t}$.\\
\\
Let us start finding a lower bound for $I_{t}$.
\begin{prop}\label{prop:subsolution}Let $u_{t}$ be the solution to Equation (\ref{eq:PDE}) for $q=tq_{1}$. Then
\[
	u_{t}>\frac{1}{2}\log(t\|q_{1}\|_{h}) \ .
\]
In particular, $I_{t}>t|q_{1}|$. 
\end{prop}
\begin{proof}The main idea of the proof lies on the fact that $\frac{1}{4}\log(\|tq_{1}\|^{2}_{h})$ is a solution to Equation (\ref{eq:PDE}), outside the zeros of $q_{1}$. To be precise, let $s_{t}:S\setminus q_{1}^{-1}(0) \rightarrow \R$ be the function such that
\[
	e^{2s_{t}}h=t|q_{1}|
\]
at every point. Then, outside the zeros of $q_{1}$, we have
\begin{align*}
	\Delta_{h}s_{t}&= h^{-1}\bar{\partial}\partial\log(\|tq_{1}\|_{h}^{2})= h^{-1}\bar{\partial}\partial[\log(t^{2}q_{1}\bar{q}_{1})-\log(h^{2})]\\
			&=-\frac{1}{2}\Delta_{h}\log(h)=K_{h}
\end{align*}
and
\[
	e^{2s_{t}}-t^{2}e^{-2s_{t}}\|q_{1}\|_{h}^{2}=t\|q_{1}\|_{h}-t\|q_{1}\|_{h}=0 \ ,
\]
hence $s_{t}$ is a solution of Equation (\ref{eq:PDE}) outside the zeros of $q_{1}$. We observe, moreover, that at the zeros of $q_{1}$, $s_{t}$ tends to $-\infty$. Therefore, the function
\[
	u^{-}=\begin{cases} 
			\frac{1}{4}\log(\|tq_{1}\|^{2}_{h}) \ \ \ \ \ \ \  \|tq_{1}\|_{h}\geq 1 \\
			0 \ \ \ \ \ \ \ \ \ \ \ \ \ \ \ \ \ \ \ \ \ \ \|tq_{1}\|_{h} < 1
		\end{cases}
\]
is a subsolution for Equation (\ref{eq:PDE}), being it the maximum of two subsolutions, and we deduce, in particular, that $u_{t}>s_{t}$.\\
\indent Now, the strong maximum principle (\cite[Thereom 2.3.1]{Jost}) implies that on any domain where $s_{t}$ is continuous up to the boundary, we have either $u_{t}>s_{t}$ or $u_{t}\equiv s_{t}$. Thus if $u_{t}(p)=s_{t}(p)$ for some $p \in S$ (and clearly $p$ cannot be a zero for $q_{1}$ in this case), then $u_{t}$ and $s_{t}$ must agree in the complement of the zeros of $q_{1}$, but this is not possible, since $s_{t}$ diverges to $-\infty$ near the zeros, whereas $u_{t}$ is smooth everywhere on $S$.\\
\indent In particular, we deduce that $I_{t}=e^{2u_{t}}h>e^{2s_{t}}h=t|q_{1}|$.
\end{proof}

\begin{cor}\label{cor:bound_lambda} Let $\lambda_{t}$ be the positive principal curvature of the maximal surface $\Sigma_{t}$, then $\lambda_{t}<1$.
\end{cor}
\begin{proof}Recall that the shape operator of $\Sigma_{t}$ can be written as
\[
	B_{t}=I_{t}^{-1}II_{t}=e^{-2u_{t}}h^{-1}\Re(tq_{1}) \ .
\]
Therefore, $\lambda_{t}^{2}=-\det(B_{t})=e^{-4u_{t}}t^{2}\|q_{1}\|_{h}^{2}<1$, by the previous proposition.
\end{proof}

In order to find an upper bound for $I_{t}$, we introduce a new metric on the surface $S$. Let $U$ be a neighbourhood of the zeros of $q_{1}$. We consider a smooth metric $g$ on $S$ in the conformal class of $h$ such that $g=|q_{1}|$ in the complement of $U$ and $\|q_{1}\|_{g}^{2}\leq 1$ everywhere on $S$. This is possible because $\|q_{1}\|_{g}^{2}=1$ on $S\setminus U$ and it vanishes at the zeros of $q_{1}$. Let $w_{t}$ be half of the logarithm of the density of $I_{t}$ with respect to $g$, i.e $w_{t}:S \rightarrow \R$ satisfies
\[
	e^{2w_{t}}g=I_{t} \ .
\]
Comparing the curvature $K_{I_{t}}$ of the metric $I_{t}$ and the curvature of the metric $e^{2w_{t}}g$ (see Section \ref{subsec:para_max}), we deduce that $w_{t}$ is the solution to Equation (\ref{eq:PDE}), where the background metric is now $g$, i.e. 
\[
	\Delta_{g}w_{t}=e^{2w_{t}}-t^{2}e^{-2w_{t}}\|q_{1}\|_{g}^{2}+K_{g} \ .
\]
We can give an upper-bound to the induced metric $I_{t}$ by estimating the function $w_{t}$.

\begin{prop}\label{prop:supersol}Let $K$ be the minimum of the curvature of $g$ and let $S_{t}$ be the positive root of the polynomial $r_{t}(x)=x^{2}+Kx-t^{2}$. Then $e^{2w_{t}}\leq S_{t}$.
\end{prop}
\begin{proof}By compactness of $S$, the function $w_{t}$ has maximum at some point $p \in S$. At that point, we have
\begin{align*}
	0\geq \Delta_{g}w_{t}(p)&=e^{2w_{t}(p)}-t^{2}e^{-2w_{t}(p)}\|q_{1}(p)\|_{g}^{2}+K_{g}(p)\\
		&=e^{-2w_{t}(p)}(e^{4w_{t}(p)}+e^{2w_{t}(p)}K_{g}(p)-t^{2}\|q_{1}(p)\|_{g}^{2})\\
		&\geq e^{-2w_{t}(p)}(e^{4w_{t}(p)}+Ke^{2w_{t}(p)}-t^{2})=e^{-2w_{t}(p)}r_{t}(e^{2w_{t}(p)})
\end{align*}
The biggest possible value in which this inequality is true is for $e^{2w_{t}(p)}=S_{t}$. Since $p$ is a point of maximum of $w_{t}$
we deduce that $e^{2w_{t}}\leq S_{t}$ everywhere on $S$.
\end{proof}

\begin{cor}Along a ray $tq_{1}$, the induced metric $I_{t}$ on the maximal surface satisfies
\[
	I_{t}=t|q_{1}|(1+o(1)) \ \ \ \ \  \ \text{for} \ t \to +\infty
\]
outside the zeros of $q_{1}$.
\end{cor}
\begin{proof}Combining Proposition \ref{prop:subsolution} and Proposition \ref{prop:supersol} we have
\[
	t|q_{1}| \leq I_{t} \leq S_{t}g \ .
\]
Now, we notice that $\frac{S_{t}}{t}$ is the biggest positive root of the polynomial $\tilde{r}_{t}(x)=x^{2}+\frac{K}{t}x-1$, hence 
\[
	\frac{S_{t}}{t} \to 1 \ \ \ \ \ \ \ \text{when} \ t \to +\infty \ .
\]
Moreover, outside of $U$, by definition $g=|q_{1}|$, thus
\[
	|q_{1}| \leq \frac{I_{t}}{t} \leq \frac{S_{t}}{t}|q_{1}| \xrightarrow{t \to +\infty} |q_{1}|
\]
and the proof is complete.
\end{proof}

We can actually be more precise about the way the induced metrics $\frac{I_{t}}{t}$ converge to the flat metric $|q_{1}|$.

\begin{lemma}\label{lm:derivative}Let $u_{t}$ be the solution to Equation (\ref{eq:PDE}) along the ray $tq_{1}$. Then, for every $t_{0}\geq 0$, we have $\dot{u}_{t_{0}} \geq 0$
everywhere on $S$. 
\end{lemma}
\begin{proof}Along the ray $tq_{1}$, Equation (\ref{eq:PDE}) can be re-written as
\begin{equation}\label{eq:PDEray}
	\Delta_{h}u_{t}=e^{2u_{t}}-e^{-2u_{t}}t^{2}\|q_{1}\|^{2}_{h}-1 \ .
\end{equation}
Taking the derivative at $t=t_{0}$ we obtain
\begin{equation}\label{eq:derivativePDE}
	\Delta_{h}\dot{u}_{t_{0}}=2e^{2u_{t_{0}}}\dot{u}_{t_{0}}-2t_{0}\|q_{1}\|_{h}^{2}e^{-2u_{t_{0}}}+2t_{0}^{2}\dot{u}_{t_{0}}e^{-2u_{t_{0}}}\|q_{1}\|^{2}_{h} \ .
\end{equation}
At a point $p$ of minimum for $\dot{u}_{t_{0}}$ we have
\[
	0\leq \Delta_{h}\dot{u}_{t_{0}}(p)=2\dot{u}_{t_{0}}(p)(e^{2u_{t_{0}}(p)}+e^{-2u_{t_{0}}(p)}t_{0}^{2}\|q_{1}(p)\|^{2}_{h})-2t_{0}\|q_{1}(p)\|^{2}_{h}e^{-2u_{t_{0}}(p)} 
\]
which implies, since $t_{0}\geq 0$, that $\dot{u}_{t_{0}}(p)\geq 0$. Hence, $\dot{u}_{t_{0}}\geq 0$ everywhere on $S$. 
\end{proof}

\begin{prop}\label{prop:monotonia} Outside the zeros of $q_{0}$, 
\[
	\frac{I_{t}}{t} \to |q_{1}| \ \ \ \text{when} \ \ \ t \to +\infty
\]
monotonically from above.
\end{prop}
\begin{proof}Recall that we can write $I_{t}=e^{2u_{t}}h$, where $u_{t}$ is the solution of Equation (\ref{eq:PDE}) for $q=tq_{1}$. By Proposition \ref{prop:subsolution}, we know that 
\[
	u_{t}>\frac{1}{2}\log(t\|q_{1}\|_{h}) \ .
\]
It is thus sufficient to show that $\varphi_{t}=u_{t}-\frac{1}{2}\log(t\|q_{1}\|_{h})>0$ is monotone decreasing in $t$.
Outside the zeros of $q_{1}$, the function $\varphi_{t}$ satisfies the differential equation
\begin{align*}
	\Delta_{h}\varphi_{t}&=\Delta_{h}u_{t}-\frac{1}{2}\Delta_{h}\log(t\|q_{1}\|_{h})
		=e^{2u_{t}}-t^{2}\|q_{1}\|_{h}^{2}e^{-2u_{t}}\\
		&=t\|q_{1}\|_{h}(e^{2u_{t}}t^{-1}\|q_{1}\|_{h}^{-1}-t\|q_{1}\|_{h}e^{-2u_{t}})\\
		&=t\|q_{1}\|_{h}(e^{2\varphi_{t}}-e^{-2\varphi_{t}})
		=2t\|q_{1}\|_{h}\sinh(2\varphi_{t}) \ .
\end{align*}
Taking the derivative at $t=t_{0}$, we obtain
\begin{equation}\label{eq:dot}
	\Delta_{h}\dot{\varphi}_{t_{0}}=2\|q_{1}\|_{h}\sinh(2\varphi_{t_{0}})+4t_{0}\|q_{1}\|_{h}\cosh(2\varphi_{t_{0}})\dot{\varphi}_{t_{0}} \ .
\end{equation}
\indent We would like to evaluate Equation (\ref{eq:dot}) at a point of maximum, but up to now the function $\dot{\varphi}_{t}$ is defined only on the complement of the zeros of $q_{1}$, and may be unbounded. However, since $e^{2u_{t}}e^{-2\varphi_{t}}=t\|q_{1}\|_{h}$, taking the derivative in $t=t_{0}$ we deduce that
\[
	2\|q_{1}\|_{h}t_{0}(\dot{u}_{t_{0}}-\dot{\varphi}_{t_{0}})=\|q_{1}\|_{h} \ ,
\]
hence, outside the zeros of $q_{1}$, we have
\[
	\dot{\varphi}_{t_{0}}=\dot{u}_{t_{0}}-\frac{1}{2t_{0}} \ ,
\]
which implies that $\dot{\varphi}_{t_{0}}$ extends to a smooth function at the zeros of $q_{1}$ because $\dot{u}_{t_{0}}$ does and, moreover, they share the same points of maximum and minimum. \\
\indent In particular, we can show that $\dot{\varphi}_{t_{0}}$ does not assume maximum at a point $p$ which is a zero of $q_{1}$. Otherwise, this would be also a point of maximum for $\dot{u}_{t_{0}}$ and we would have 
\begin{align*}
	0\geq \Delta_{h}\dot{u}_{t_{0}}(p)&=2e^{2u_{t_{0}}(p)}\dot{u}_{t_{0}}(p)-2t_{0}\|q_{0}(p)\|_{h}^{2}e^{-2u_{t_{0}}(p)}+2t_{0}^{2}\dot{u}_{t_{0}(p)}e^{-2u_{t_{0}}(p)}\|q_{0}(p)\|_{h}^{2}\\
		&=2e^{2u_{t_{0}}(p)}\dot{u}_{t_{0}}(p)
\end{align*}
which would imply that $\dot{u}_{t_{0}}\leq 0$, and, together with Lemma \ref{lm:derivative}, this would imply that $\dot{u}_{t_{0}}$ should vanish identically. But then
\[
	0=\Delta_{h}\dot{u}_{t_{0}}=-2t_{0}\|q_{1}\|_{h}^{2}e^{-2u_{t_{0}}}
\]
would give a contradiction. \\
\indent Therefore, $\dot{\varphi}_{t_{0}}$ takes maximum at a point $p$ outside the zeros of $q_{1}$, and we have
\begin{align*}
	0\geq \Delta_{h}\dot{\varphi}_{t_{0}}(p)&=2\|q_{1}(p)\|_{h}\sinh(2\varphi_{t_{0}}(p))+4t_{0}\|q_{1}(p)\|_{h}\cosh(2\varphi_{t_{0}}(p))\dot{\varphi}_{t_{0}}(p)\\
		&>4t_{0}\|q_{1}(p)\|_{h}\cosh(2\varphi_{t_{0}}(p))\dot{\varphi}_{t_{0}}(p)>4t_{0}\|q_{1}(p)\|_{h}\dot{\varphi}_{t_{0}}(p) \ ,
\end{align*}
which implies that $\dot{\varphi}_{t_{0}}<0$ everywhere on $S$, and $\varphi_{t}$ is monotone decreasing in $t$ as desired.
\end{proof}

\begin{cor}\label{cor:monotonia}Let $\lambda_{t}$ be the positive principal curvature of the maximal surface $\Sigma_{t}$. Then $\lambda_{t} \to 1$ monotonically outside the zeros of $q_{1}$, when $t$ goes to $+\infty$
\end{cor}
\begin{proof}Recall that the shape operator of $\Sigma_{t}$ can be written as
\[
	B_{t}=I_{t}^{-1}II_{t}=e^{-2u_{t}}h^{-1}\Re(tq_{1}) \ .
\]
Therefore, $\lambda_{t}^{2}=-\det(B_{t})=e^{-4u_{t}}t^{2}\|q_{1}\|_{h}^{2}$ and this is monotonically increasing to $1$ by the previous proposition.
\end{proof}

\subsection{Asymptotics and rigidity of the Lorentzian Hausdorff dimension}
We now compare the Lorentzian Hausdorff dimension of the limit set of a GHMC anti-de Sitter manifold with the critical exponent of the 
unique maximal surface.

\begin{lemma}[\cite{GM_Hdim}]\label{lm:comparison}Let $\rho$ be the holonomy representation of a GHMC anti-de Sitter manifold $M$ with limit set $\Lambda_{\rho}$. Then
\[
	\LHdim(\Lambda_{\rho}) \leq \delta(\rho) \ .
\]
\end{lemma}
\begin{proof}Let $\Sigma$ be the unique maximal surface embedded in $M$. We identify the universal cover of $M$ with the domain of dependence $\mathcal{D}(\Lambda_{\rho})$ of the limit set. In this way, $\Sigma$ is lifted to a minimal disc $\tilde{\Sigma}$ in $AdS_{3}$ with asymptotic boundary $\Lambda_{\rho}$, contained in the convex hull $\mathcal{C}(\Lambda_{\rho})$. We fix a base point $x_{0}\in \tilde{\Sigma}$. By definition,
\[
	\delta(\rho)=\limsup_{R \to +\infty}\frac{1}{R}\log(\#\{ \gamma \in \pi_{1}(S) \ | \ d_{I}(\rho(\gamma)x_{0}, x_{0})\leq R \}) \ ,  
\]
where $I$ is the induced metric on $\tilde{\Sigma}$, and by Theorem \ref{thm:equivalence}
\[
	\LHdim(\Lambda_{\rho})=\limsup_{R \to +\infty}\frac{1}{R}\log(\#\{\gamma \in \pi_{1}(S) \ | \ d_{AdS}(\rho(\gamma)x_{0}, x_{0}) \leq R\}) .
\]
Therefore, it is sufficient to show that for every pair of points $x,y \in \tilde{\Sigma}$, we have
\[
	d_{I}(x,y) \leq d_{AdS}(x,y) \ .
\]
Since $\Sigma$ is a Cauchy surface for $M$, the geodesic connecting $x$ and $y$ is space-like. We can thus find a Lorentzian plane $P\subset AdS_{3}$ containing $x$ and $y$. In an affine chart, this is isometric to $(\R \times (-\pi/2, \pi/2), dt^{2}-\cosh^{2}(t)ds^{2})$, where $t$ is the arc-length parameter of the space-like geodesic between $x$ and $y$. By intersecting $P$ with $\tilde{\Sigma}$ we obtain a curve $\gamma \subset \tilde{\Sigma}$ with length
\[
	\length(\gamma)=\int_{0}^{d_{AdS}(x,y)}\sqrt{1-\cosh^{2}(t)s'(t)}dt \leq d_{AdS}(x,y) \ .
\]
As a consequence, the distance between $x$ and $y$ in the induced metric of $\tilde{\Sigma}$ must be smaller than $d_{AdS}(x,y)$.
\end{proof}

\begin{teo}\label{thm:asy_entropy}Let $M_{t}$ be the sequence of GHMC anti-de Sitter manifolds parameterised by the ray $(h, tq_{1}) \in T^{*}\T(S)$ for some non-zero holomorphic quadratic differential $q_{1}$. Let $\Lambda_{t}$ be the limit sets of the corresponding holonomy representations. Then
\[
	\lim_{t \to +\infty}\LHdim(\Lambda_{t})=0
\]
\end{teo}
\begin{proof}By Lemma \ref{lm:comparison}, it is sufficient to show that the critical exponent relative to the induced metric on the maximal surface tends to $0$ when $t$ goes to $+\infty$. Since the metrics $I_{t}=e^{2u_{t}}h$ are bounded from below by the flat metrics with conical singularities $g_{t}=t|q_{1}|$ (Proposition \ref{prop:subsolution}), we deduce that
\[
	\delta(\rho_{t})\leq \delta(g_{t}) \ .
\]
The proof is then completed by noticing that $\delta(t|q_{1}|)=t^{-1}\delta(|q_{1}|)$. 
\end{proof}

\indent In order to prove a rigidity result for the critical exponent relative to the induced metric on the maximal surface and the Lorentzian Hausdorff dimension, we study the derivative of the critical exponent along a ray. To this aim, we need the following useful formula for its variation along a path of smooth Riemannian metrics:
\begin{teo}[\cite{KKW}]\label{thm:variation_entropy} Let $g_{t}$ be a smooth path of negatively curved Riemannian metrics on a closed manifold $S$. Then 
\[
	\frac{d}{dt}\delta(g_{t})_{|_{t=t_{0}}}=-\frac{\delta(g_{t_{0}})}{2}\int_{T^{1}S}\frac{d}{dt}g_{t}(v,v)_{|_{t=t_{0}}}d\mu_{t_{0}}
\]
where $\mu_{t_{0}}$ denotes the Bowen-Margulis measure on the unit tangent bundle $T^{1}S$ of $(S,g_{t_0})$.
\end{teo}

\begin{prop}\label{prop:entropy_decresce}The critical exponent relative to the induced metric on the maximal surface of a GHMC anti-de Sitter manifold is decreasing along a ray $tq_{1}$ for $t\geq 0$. 
\end{prop}
\begin{proof}
The induced metrics on the maximal surfaces are $I_{t}=e^{2u_{t}}h$, thus, using Lemma \ref{lm:derivative}, for every unit tangent vector $v \in T^{1}S$ 
\[
	\frac{d}{dt}I_{t}(v,v)_{|_{t=t_{0}}}=2\dot{u}_{t_{0}}e^{2u_{t_{0}}}h(v,v) \geq 0 \ .
\]
Since the induced metrics $I_{t}$ are negatively curved by the Gauss equation and Corollary \ref{cor:bound_lambda}, we can apply Theorem \ref{thm:variation_entropy} and deduce that the critical exponent is non increasing.\\
\indent To prove that it is decreasing, we notice that 
\[
	\frac{d}{dt}\delta(I_{t})_{|_{t=t_{0}}}=-\frac{\delta(I_{t_{0}})}{2}\int_{T^{1}S}\frac{d}{dt}I_{t}(v,v)_{|_{t=t_{0}}}d\mu_{t_{0}}=0
\]
if and only if $\dot{u}_{t_{0}}$ vanishes identically on $S$, because the Bowen-Margulis measure is positive in all non empty open sets. But in this case, Equation (\ref{eq:derivativePDE}) reduces to 
\[
	0=2t_{0}\|q_{1}\|_{h}^{2}e^{-2u_{t_{0}}}
\]
which implies that $t_{0}=0$, because $q_{1}$ is not identically zero. 
\end{proof}

Recall that we denoted by $\delta(h,q)$ the critical exponent relative to the induced metric on the maximal surface embedded in the GHMC anti-de Sitter manifold corresponding to $(h,q) \in T^{*}\T(S)$. 

\begin{cor}\label{cor:maximum}$\delta(h,q)\leq 1$ for every $(h,q) \in T^{*}\T(S)$ and $\delta(h,q)=1$ if and only if $q=0$.
\end{cor}
\begin{proof}If $q=0$, the function $u=0$ is the unique solution to Equation (\ref{eq:PDE}). Hence, the induced metric on the maximal surface is hyperbolic, and it is well-known that the critical exponent of the hyperbolic metric is $1$. \\ 
\indent On the other hand, since the function $\delta(h,tq_{1})$ is decreasing for $t\geq 0$, for every non-zero quadratic differential $q$ we have $\delta(h,q)<\delta(h, 0)=1$.
\end{proof}

We then obtain a new proof of the rigidity of the Lorentzian Hausdorff dimension (\cite{GM_Hdim}):
\begin{teo}Let $M$ be a GHMC anti-de Sitter manifold and let $\Lambda$ be its limit set. Then
\[
	\LHdim(\Lambda)=1
\]
if and only if $M$ is Fuchsian.
\end{teo}
\begin{proof}If $M$ is Fuchsian, the holonomy representation $\rho=(\rho_{0}, \rho_{0})$ preserves the totally geodesic space-like plane $P_{0}$, that is isometric to the hyperbolic plane. Fix the base point $x_{0}$ on $P_{0}$. Since for every $\gamma \in \pi_{1}(S)$, the isometry $\rho(\gamma)$ acts on the plane $P_{0}$ like the hyperbolic isometry $\rho_{0}(\gamma)$ on $\h^{2}$ (see Section \ref{sec:holder}), the critical exponent relative to $d_{AdS}$ coincides with the critical exponent relative to the hyperbolic metric associated to $\rho_{0}$, which is equal to $1$.\\
\indent Viceversa, suppose that $\LHdim(\Lambda)=1$, then by Lemma \ref{lm:comparison} the critical exponent relative to the induced metric on the maximal surface embedded in $M$ is at least $1$. By Corollary \ref{cor:maximum}, we deduce that $M$ is Fuchsian. 
\end{proof}

\section{Width of the convex core}\label{sec:width}
Another geometric quantity associated to GHMC anti-de Sitter manifolds is the width of the convex core. This has already been extensively studied in \cite{seppimaximal}. Combining the aformentioned work with our estimates in Section \ref{sec:entropy}, we can describe its asymptotic behaviour.\\
\\
\indent We recall that the convex core of a GHMC anti-de Sitter manifold $M$ is homeomorphic to $S \times [0,1]$, unless $M$ is Fuchsian, in which case it is a totally geodesic surface. The width of the convex core expresses how far $M$ is from being Fuchsian, as it measures the distance between the two boundary components of the convex core. More precisely, let $\Lambda_{\rho}$ be the limit set of the holonomy representation $\rho$ of $M$. The convex core can be realised as the quotient of the convex hull of $\Lambda_{\rho}$ in $AdS_{3}$ by the action of $\rho(\pi_{1}(S))$. 

\begin{defi}\label{def:width} The width $w(M)$ of the convex core of $M$ is the supremum of the length of a time-like geodesic contained in $\mathcal{C}(\Lambda_{\rho})$. 
\end{defi}

We can give an equivalent definition by introducting a time-like distance in $AdS_{3}$. Given two points $x,y \in AdS_{3}$, we denote with $\gamma_{x,y}$ the unique geodesic connecting the two points. We define
\[
	d_{t}: AdS_{3} \times AdS_{3} \rightarrow \R_{\geq 0}
\]
as
\[
	d_{t}(x,y)=\begin{cases} \length(\gamma_{x,y}) \ \ \ \ \ \ \ \  \text{if $\gamma_{x,y}$ is time-like}\\
					0 \ \ \ \ \ \ \ \ \ \ \ \ \ \ \ \ \ \ \ \ \ \ \text{otherwise}	
		      \end{cases}
\]
where the length of a time-like curve $\gamma:[0,1] \rightarrow AdS_{3}$ is 
\[
	\length(\gamma)=\int_{0}^{1}\sqrt{-\|\dot{\gamma}(t)\|^{2}}dt \ ,
\]
where $\|\dot{\gamma}(t)\|$ denotes the norm of the tangent vector $\dot{\gamma}(t) \in T_{\gamma(t)}AdS_{3}$ with respect to the Lorentzian metric on $AdS_{3}$.
Therefore, Definition \ref{def:width} is equivalent to 
\[
	w(M)=\sup_{\substack{p \in \mathcal{C}(M)^{+} \\ q \in \mathcal{C}(M)^{-}}}d_{t}(p,q)
\]
where $\mathcal{C}(M)^{\pm}$ denotes the upper- and lower-boundary of the convex core. Notice, in particular, that $w(M)=0$, if and only if $M$ is Fuchsian.\\
\\
\indent Seppi found an estimate for the width of the convex core in terms of the principal curvatures of the maximal surface:

\begin{teo}[Theorem 1.B \cite{seppimaximal}] \label{teo:stimawidth} There exist universal constants $C>0$ and $\delta \in (0,1)$ such that if $\Sigma$ is a maximal surface in a GHMC anti-de Sitter manifold with principal curvatures $\lambda$ satisfying $\delta \leq \| \lambda \|_{\infty}<1$, then
\[
	\tan(w(M))\geq \left( \frac{1}{1-\|\lambda\|_{\infty}}\right)^{\frac{1}{C}} \ .
\]
\end{teo}

We consider now a family of GHMC anti-de Sitter manifolds $M_{t}$ parameterised by the ray $(h,tq_{1})\in T^{*}\T(S)$ for a non-zero holomorphic quadratic differential $q_{1}$. 

\begin{prop}The width of the convex core $w(M_{t})$ converges to $\pi/2$ when $t$ goes to $+\infty$.
\end{prop}
\begin{proof}By Theorem \ref{teo:stimawidth}, it is sufficient to show that the positive principal curvature $\lambda_{t}$ of the maximal surface $\Sigma_{t}$ embedded in $M_{t}$ converges to $1$. This is exactly the content of Corollary \ref{cor:monotonia} 
\end{proof}

\bibliographystyle{alpha}
\bibliographystyle{ieeetr}
\bibliography{bs-bibliography}

\bigskip

\noindent \footnotesize \textsc{DEPARTMENT OF MATHEMATICS, UNIVERSITY OF LUXEMBOURG}\\
\emph{E-mail address:}  \verb|andrea.tamburelli@uni.lu|

\end{document}